\documentclass[a4paper]{amsart}
\usepackage{amscd}
\usepackage{amsmath}
\usepackage{mathrsfs}
\usepackage{amssymb}
\usepackage{amsthm}
\usepackage{bbm}
\usepackage{stmaryrd}

\usepackage[T1]{fontenc}

\makeatletter
\@namedef{subjclassname@2020}{%
  \textup{2020} Mathematics Subject Classification}
\makeatother

\newif\ifpdf
\ifx\pdfoutput\undefined
   \pdffalse        % we are not running PDFLaTeX
\else
   \pdfoutput=1     % we are running PDFLaTeX
   \pdftrue
\fi

\ifpdf
   \usepackage[pdftex]{graphicx}
\else
   \usepackage{graphicx}
\fi

\numberwithin{equation}{section} \swapnumbers

\newtheorem{satz}{Satz}[section]

\newtheorem{theorem}[satz]{Theorem}
\newtheorem{proposition}[satz]{Proposition}
\newtheorem{corollary}[satz]{Corollary}
\newtheorem{lemma}[satz]{Lemma}
\newtheorem{assumption}[satz]{Assumption}

\newtheorem{definition}[satz]{Definition}

\newtheorem{remark}[satz]{Remark}

\newtheorem{example}[satz]{Example}

\newcommand{\bbb}{\mathbb{B}}
\newcommand{\bbr}{\mathbb{R}}
\newcommand{\bbe}{\mathbb{E}}
\newcommand{\bbn}{\mathbb{N}}
\newcommand{\bbp}{\mathbb{P}}

\newcommand{\bbw}{\mathbb{W}}
\newcommand{\bbv}{\mathbb{V}}

\newcommand{\cala}{\mathscr{A}}
\newcommand{\calb}{\mathscr{B}}
\newcommand{\cald}{\mathscr{D}}
\newcommand{\cale}{\mathscr{E}}
\newcommand{\calf}{\mathscr{F}}
\newcommand{\calh}{\mathscr{H}}

\newcommand{\calp}{\mathscr{P}}

\newcommand{\ran}{{\rm ran}}

\newcommand{\Id}{{\rm Id}}

\newcommand{\sgn}{{\rm sgn}}

\newcommand{\la}{\langle}
\newcommand{\ra}{\rangle}

\newcommand{\bbI}{\mathbbm{1}}

\begin{document}

\title[The dual Yamada-Watanabe theorem for SPDEs]{The dual Yamada-Watanabe theorem for mild solutions to stochastic partial differential equations}
\author{Stefan Tappe}
\address{Albert Ludwig University of Freiburg, Department of Mathematical Stochastics, Ernst-Zermelo-Stra\ss{}e 1, D-79104 Freiburg, Germany}
\email{stefan.tappe@math.uni-freiburg.de}
\date{29 June, 2021}
\thanks{I am grateful to Mikhail Urusov for bringing my attention to this research topic, and for valuable discussions. I am also grateful to the referee for helpful remarks. Furthermore, I gratefully acknowledge financial support from the Deutsche Forschungsgemeinschaft (DFG, German Research Foundation) -- project number 444121509.}
\begin{abstract}
We provide the dual result of the Yamada-Watanabe theorem for mild solutions to semilinear stochastic partial differential equations with path-dependent coefficients. An essential tool is the so-called ``method of the moving frame'', which allows us to reduce the proof to infinite dimensional stochastic differential equations.
\end{abstract}

\maketitle

\section{Introduction}\label{sec-intro}

The Yamada-Watanabe theorem, which goes back to \cite{Yamada}, has been generalized into several directions, also for stochastic differential equations (SDEs) in infinite dimension. Such generalizations can be found, for example, in \cite{Prevot-Roeckner}, \cite{BLP}, \cite{Roeckner} (see also Appendix E in \cite{Liu-Roeckner}), \cite{Zhao}, \cite{Ondrejat} and \cite{Tappe-YW} as well as \cite{Kurtz-1} and the successive paper \cite{Kurtz-2}.

The dual result of the Yamada-Watanabe theorem has originally been presented by Cherny in \cite{Cherny}. This result has been generalized to SDEs driven by Poisson processes (see \cite{YHX}), to SDEs driven by L\'{e}vy processes (see \cite{Criens}), to stochastic evolution equations in the framework of the variational approach with path-dependent coefficients (see \cite{Qiao} and \cite{Rehmeier}) and to stochastic partial differential equations (SPDEs) in the framework of the semigroup approach with state-dependent coefficients in \cite{Ondrejat}. We also mention the recent preprint \cite{Criens-Ritter}, where this result has been generalized to analytically weak solutions of Banach space valued semilinear SPDEs with deterministic initial values.

The goal of the present paper is to provide the dual result of the Yamada-Watanabe theorem for SPDEs in the framework of the semigroup approach with path-dependent coefficients. For this purpose, consider a semilinear SPDE
\begin{align}\label{SPDE}
dX(t) = (A X(t) + \alpha(t,X)) dt + \sigma(t,X)dW(t).
\end{align}
Equations of the type (\ref{SPDE}) have been investigated, for example, in \cite{Da_Prato} and \cite{Atma-book}. Here $A$ is the generator of a $C_0$-semigroup $(S_t)_{t \geq 0}$ on the state space $H$ of the SPDE (\ref{SPDE}), which is assumed to be a separable Hilbert space. The coefficients $\alpha$ and $\sigma$ are path-dependent and satisfy standard measurability and adaptedness conditions; see Assumption \ref{ass-1}. Furthermore, $W$ is a cylindrical Wiener process. We assume that the semigroup $(S_t)_{t \geq 0}$ can be extended to a $C_0$-group on a larger Hilbert space $\calh$; see Assumption \ref{ass-2}. We remark that this assumption is satisfied for every \emph{pseudo-contractive}\footnote{The notion \emph{quasi-contractive} is also used in the literature.} semigroup. Under the previous assumptions, we will prove the following results. Concerning the precise definitions of the upcoming solution and uniqueness concepts, we refer to Section \ref{sec-framework}. 

\begin{theorem}\label{thm-1}
Let $\mu$ be a probability measure on $(H,\calb(H))$. Suppose there exists a mild solution $(X,W)$ to the SPDE (\ref{SPDE}) such that $\mu$ is the distribution of $X(0)$, and that joint uniqueness in law given $\mu$ holds for (\ref{SPDE}). Then pathwise uniqueness given $\mu$ holds for (\ref{SPDE}) as well.
\end{theorem}

\begin{corollary}\label{cor-1}
Suppose that the SPDE (\ref{SPDE}) has a mild solution, and that joint uniqueness in law holds for (\ref{SPDE}). Then pathwise uniqueness holds for (\ref{SPDE}) as well.
\end{corollary}

\begin{theorem}\label{thm-2}
Suppose that the semigroup $(S_t)_{t \geq 0}$ can be extended to a $C_0$-group $(U_t)_{t \in \bbr}$ on the state space $H$. Then for every $x \in H$ the following statements are equivalent:
\begin{enumerate}
\item[(i)] Uniqueness in law given $\delta_x$ holds for the SPDE (\ref{SPDE}).

\item[(ii)] Joint uniqueness in law given $\delta_x$ holds for the SPDE (\ref{SPDE}).
\end{enumerate}
\end{theorem}

\begin{corollary}\label{cor-2}
Suppose that the semigroup $(S_t)_{t \geq 0}$ can be extended to a $C_0$-group $(U_t)_{t \in \bbr}$ on the state space $H$. Then the following statements are equivalent:
\begin{enumerate}
\item[(i)] $\delta$-uniqueness in law holds for the SPDE (\ref{SPDE}).

\item[(ii)] Joint $\delta$-uniqueness in law holds for the SPDE (\ref{SPDE}).
\end{enumerate}
\end{corollary}

Of course, Corollary \ref{cor-1} is an immediate consequence of Theorem \ref{thm-1}, and Corollary \ref{cor-2} is an immediate consequence of Theorem \ref{thm-2}. Furthermore, we obtain the following extension of the Yamada-Watanabe theorem for semilinear SPDEs. Its proof is essentially a consequence of Theorem \ref{thm-1} and the Yamada-Watanabe theorem from \cite{Tappe-YW}.

\begin{theorem}\label{thm-3}
The following statements are equivalent:
\begin{enumerate}
\item[(i)] The SPDE (\ref{SPDE}) has a unique mild solution.

\item[(ii)] For each probability measure $\mu$ on $(H,\calb(H))$ there exists a mild solution $(X,W)$ to (\ref{SPDE}) such that $\mu$ is the distribution of $X(0)$, and pathwise uniqueness for (\ref{SPDE}) holds.

\item[(iii)] For each probability measure $\mu$ on $(H,\calb(H))$ there exists a martingale solution $(X,W)$ to (\ref{SPDE}) such that $\mu$ is the distribution of $X(0)$, and pathwise uniqueness for (\ref{SPDE}) holds.

\item[(iv)] For each probability measure $\mu$ on $(H,\calb(H))$ there exists a mild solution $(X,W)$ to (\ref{SPDE}) such that $\mu$ is the distribution of $X(0)$, and joint uniqueness in law for (\ref{SPDE}) holds.
\end{enumerate}
\end{theorem}

The next result can be regarded as a dual statement to the Yamada-Watanabe theorem. Its proof immediately follows from Corollary \ref{cor-2} and Theorem \ref{thm-1}.

\begin{theorem}\label{thm-4}
Suppose that the semigroup $(S_t)_{t \geq 0}$ can be extended to a $C_0$-group $(U_t)_{t \in \bbr}$ on the state space $H$. If the SPDE (\ref{SPDE}) has a mild solution and $\delta$-uniqueness in law holds, then $\delta$-pathwise uniqueness holds for (\ref{SPDE}) as well.
\end{theorem}

The main idea for proving Theorems \ref{thm-1} and \ref{thm-2} is as follows. We consider the infinite dimensional SDE
\begin{align}\label{SDE}
dY(t) = \bar{\alpha}(t,Y) dt + \bar{\sigma}(t,Y) dW(t)
\end{align}
on the larger Hilbert space $\calh$ from Assumption \ref{ass-2}, where the coefficients $\bar{\alpha}$ and $\bar{\sigma}$ are defined by means of $\alpha$ and $\sigma$. Then the SPDE (\ref{SPDE}) and the SDE (\ref{SDE}) are connected, which is due to the ``method of the moving frame'', which has been presented in \cite{SPDE}, and which has also been used for the proof of the Yamada-Watanabe theorem in \cite{Tappe-YW}. For the SDE (\ref{SDE}) the statements of Theorems \ref{thm-1} and \ref{thm-2} hold true by virtue of \cite{Rehmeier}, and many properties transfer between the two equations (\ref{SPDE}) and (\ref{SDE}). One exception is uniqueness in law, which, in contrast to pathwise uniqueness and joint uniqueness in law, does not transfer from the SPDE (\ref{SPDE}) to the SDE (\ref{SDE}). For this reason we assume in Theorem \ref{thm-2} that the semigroup $(S_t)_{t \geq 0}$ extends to a $C_0$-group on the given state space $H$. From a practical perspective, this does not mean a severe restriction, because the solutions to (\ref{SPDE}) can always be realized on a larger Hilbert space, where this assumption is fulfilled; see Lemma \ref{lemma-larger-space}.

The remainder of this paper is organized as follows. In Section \ref{sec-framework} we present the general framework, in Section \ref{sec-SDEs} we provide the required results about infinite dimensional SDEs, and in Section \ref{sec-proof} we provide the proofs of the main results. In Section \ref{sec-examples} we provide some examples, illustrating our findings.

\section{Framework and definitions}\label{sec-framework}

In this section, we prepare the required framework and definitions.
Let $H$ be a separable Hilbert space and let $(S_t)_{t \geq 0}$ be a $C_0$-semigroup on $H$ with infinitesimal generator $A : \cald(A) \subset H \rightarrow H$. The path space $\mathbb{W}(H) := C(\mathbb{R}_+;H)$ is the space of all continuous functions from $\mathbb{R}_+$ to $H$. Equipped with the metric
\begin{align*}
\rho(w_1,w_2) := \sum_{k=1}^{\infty} 2^{-k} \Big( \sup_{t \in [0,k]} \| w_1(t) - w_2(t) \| \wedge 1 \Big),
\end{align*}
the path space $(\mathbb{W}(H),\rho)$ is a Polish space; actually a Fr\'{e}chet space. Furthermore, we define the subspace
\begin{align*}
\mathbb{W}_0(H) := \{ w \in \mathbb{W}(H) : w(0) = 0 \}
\end{align*}
consisting of all functions from the path space $\mathbb{W}(H)$ starting in zero.  Let $U$ be another separable Hilbert space and let $L_2(U,H)$ denote the space of all Hilbert-Schmidt operators from $U$ to $H$ equipped with the Hilbert-Schmidt norm. Let $\alpha : \mathbb{R}_+ \times \mathbb{W}(H) \rightarrow H$ and $\sigma : \mathbb{R}_+ \times \mathbb{W}(H) \rightarrow L_2(U,H)$ be mappings. For $t \in \mathbb{R}_+$ we denote by
$\calb_t(\mathbb{W}(H))$ the $\sigma$-algebra generated by all maps $\mathbb{W}(H) \rightarrow H$, $w \mapsto w(s)$ for $s \in [0,t]$.

\begin{assumption}\label{ass-1}
We suppose that the following conditions are satisfied:
\begin{enumerate}
\item $\alpha$ is $\calb(\mathbb{R}_+) \otimes \calb(\mathbb{W}(H)) / \calb(H)$-measurable such that for each $t \in \mathbb{R}_+$ the mapping $\alpha(t,\bullet)$ is $\calb_t(\mathbb{W}(H))/\calb(H)$-measurable.

\item $\sigma$ is $\calb(\mathbb{R}_+) \otimes \calb(\mathbb{W}(H)) / \calb(L_2(U,H))$-measurable such that for each $t \in \mathbb{R}_+$ the mapping $\sigma(t,\bullet)$ is $\calb_t(\mathbb{W}(H))/\calb(L_2(U,H))$-measurable.
\end{enumerate}
\end{assumption}

We call a filtered probability space $\mathbb{B} = (\Omega,\calf,(\calf_t)_{t \in \bbr_+},\mathbb{P})$ satisfying the usual conditions a \emph{stochastic basis}. In the sequel, we shall use the abbreviation $\mathbb{B}$ for a stochastic basis $(\Omega,\calf,(\calf_t)_{t \in \bbr_+},\mathbb{P})$, and the abbreviation $\mathbb{B}'$ for another stochastic basis $(\Omega',\calf',(\calf_t')_{t \in \bbr_+},\mathbb{P}')$.

For a sequence $(\beta_k)_{k \in \mathbb{N}}$ of independent Wiener processes we call the sequence
\begin{align*}
W = (\beta_k)_{k \in \mathbb{N}}
\end{align*}
a \emph{standard $\mathbb{R}^{\infty}$-Wiener process}. Given such a process, we fix an orthonormal basis $(e_k)_{k \in \bbn}$ of $U$. Then
\begin{align*}
\sum_{k=1}^{\infty} \beta_k \langle e_k,\cdot \rangle
\end{align*}
is an $U$-valued cylindrical Wiener process, and its covariance operator is given by the identity mapping; see \cite[Remark 7.3]{Riedle}. Hence, we can identify $W$ with the informal expression $\sum_{k=1}^{\infty} \beta_k e_k$. Now, we also fix another separable Hilbert space $\bar{U}$ and a one-to-one Hilbert Schmidt operator $J \in L_2(U,\bar{U})$. Then the process
\begin{align*}
\bar{W} := \sum_{k=1}^{\infty} \beta_k J e_k
\end{align*}
is an $\bar{U}$-valued trace class Wiener process with covariance operator $Q := J J^* \in L_1(\bar{U})$. We call $\bar{W}$ the $Q$-Wiener process associated to $W$.

\begin{definition}[Martingale solution]\label{def-martingal-sol}
A pair $(X,W)$, where $X$ is an adapted process with paths in $\mathbb{W}(H)$ and $W$ is a standard $\mathbb{R}^{\infty}$-Wiener process on some stochastic basis $\mathbb{B}$, is called a \emph{martingale solution} to (\ref{SPDE}), if we have $\mathbb{P}$-almost surely
\begin{align*}
\int_0^t \| \alpha(s,X) \| ds + \int_0^t \| \sigma(s,X) \|_{L_2(U,H)}^2 ds < \infty \quad \text{for all $t \in \mathbb{R}_+$}
\end{align*}
and $\mathbb{P}$-almost surely
\begin{align*}
X(t) = S_t X(0) + \int_0^t S_{t-s} \alpha(s,X)ds + \int_0^t S_{t-s} \sigma(s,X) dW(s) \quad \text{for all $t \in \mathbb{R}_+$.}
\end{align*}
\end{definition}

\begin{remark}
In other words, a martingale solution to (\ref{SPDE}) exists if there are a stochastic basis $\mathbb{B}$ as well as an adapted, continuous process $X$ and a standard $\mathbb{R}^{\infty}$-Wiener process $W$ defined on this stochastic basis such that $X$ is a mild solution to (\ref{SPDE}). In this case, the triplet $(\bbb,X,W)$ is also called a martingale solution to (\ref{SPDE}). For convenience, in this paper we agree to call the pair $(X,W)$ a martingale solution to the SPDE (\ref{SPDE}).
\end{remark}

\begin{remark}
By the measurability conditions from Assumption~\ref{ass-1}, the processes $\alpha(\bullet,X)$ and $\sigma(\bullet,X)$ from Definition~\ref{def-martingal-sol} are adapted.
\end{remark}

\begin{remark}\label{rem-integral-cylindrical}
The stochastic integral from Definition~\ref{def-martingal-sol} is defined as
\begin{align*}
\int_0^t S_{t-s} \sigma(s,X) dW(s) := \int_0^t S_{t-s} \sigma(s,X) \circ J^{-1} d\bar{W}(s), \quad t \in \bbr_+,
\end{align*}
where $\bar{W}$ denotes the $\bar{U}$-valued $Q$-Wiener process associated to $W$. Concerning this definition, we note that for each linear operator $\Phi \in L(U,H)$ we have $\Phi \in L_2(U,H)$ if and only if $\Phi \circ J^{-1} \in L_2(Q^{1/2}(\bar{U}),H)$, and in this case
\begin{align*}
\| \Phi \circ J^{-1} \|_{L_2(Q^{1/2}(\bar{U}),H)} = \| \Phi \|_{L_2(U,H)}, 
\end{align*}
where the separable Hilbert space $Q^{1/2}(\bar{U})$ is equipped with the inner product
\begin{align*}
\langle u,v \rangle_{Q^{1/2}(\bar{U})} := \langle Q^{-1/2} u, Q^{-1/2} v \rangle_{\bar{U}}, \quad u,v \in Q^{1/2}(\bar{U}).
\end{align*}
For further details, we refer to \cite[Sec. 2.5]{Prevot-Roeckner}.
\end{remark}

\begin{remark}
In Definition \ref{def-martingal-sol} we have followed the convention to speak about \emph{martingale solutions} rather than \emph{weak solutions} in the context of semilinear SPDEs; cf. \cite{Da_Prato} or \cite{Atma-book}.
\end{remark}

\begin{definition}[Pathwise uniqueness, Uniqueness in law, Joint uniqueness in law]
Let $\calp$ be a family of probability measures on $(H,\calb(H))$.
\begin{enumerate}
\item We say that \emph{pathwise uniqueness given $\calp$} holds for (\ref{SPDE}) if for two martingale solutions $(X,W)$ and $(X',W)$ on the same stochastic basis $\mathbb{B}$ and with the same $\mathbb{R}^{\infty}$-Wiener process $W$ such that $\mathbb{P}(X(0) = X'(0)) = 1$ and $\bbp \circ X(0)^{-1} \in \calp$ we have $\bbp$-almost surely $X = X'$.

\item We say that \emph{uniqueness in law given $\calp$} holds for (\ref{SPDE}) if for two martingale solutions $(\bbb,X,W)$ and $(\bbb',X',W')$ such that
\begin{align}\label{u-in-law-1}
\bbp \circ X(0)^{-1} = \bbp' \circ X'(0)^{-1} \in \calp
\end{align}
as measures on $(H,\calb(H))$ we have
\begin{align*}
\bbp \circ X^{-1} = \bbp' \circ (X')^{-1}
\end{align*}
as measures on $(\mathbb{W}(H),\calb(\mathbb{W}(H)))$.

\item We say that \emph{joint uniqueness in law given $\calp$} holds for (\ref{SPDE}) if for two martingale solutions $(\bbb,X,W)$ and $(\bbb',X',W')$  such that (\ref{u-in-law-1}) as measures on $(H,\calb(H))$ we have
\begin{align*}
\bbp \circ (X,\bar{W})^{-1} = \bbp' \circ (X',\bar{W'})^{-1}
\end{align*}
as measures on $(\mathbb{W}(H) \times \mathbb{W}_0(\bar{U}),\calb(\mathbb{W}(H)) \otimes \calb(\mathbb{W}_0(\bar{U})))$.
\end{enumerate}
\end{definition}

Suppose that pathwise uniqueness given $\calp$ holds for (\ref{SPDE}). If $\calp$ has one of the following structures, then we introduce further terminology:
\begin{itemize}
\item If $\calp$ consists of all probability measures on $(H,\calb(H))$, then we say that \emph{pathwise uniqueness} holds for (\ref{SPDE}).

\item If $\calp = \{ \mu \}$ for some probability measure $\mu$ on $(H,\calb(H))$, then we say that \emph{pathwise uniqueness given $\mu$} holds for (\ref{SPDE}).

\item If $\calp = \{ \delta_x : x \in H \}$ consists of all Dirac measures, then we say that \emph{$\delta$-pathwise uniqueness} holds for (\ref{SPDE}).
\end{itemize}
We agree on analogous conventions for the notions \emph{uniqueness in law given} and \emph{joint uniqueness in law}.

\begin{definition}\label{def-cale}
Let $\hat{\cale}(H)$ be the set of maps $F : H \times \mathbb{W}_0(\bar{U}) \rightarrow \mathbb{W}(H)$ such that for every probability measure $\mu$ on $(H,\calb(H))$ there exists a map
\begin{align*}
F_{\mu} : H \times \mathbb{W}_0(\bar{U}) \rightarrow \mathbb{W}(H),
\end{align*}
which is $\overline{\calb(H) \otimes \calb(\mathbb{W}_0(\bar{U}))}^{\mu \otimes \mathbb{P}^Q} / \calb(\mathbb{W}(H))$-measurable, such that for $\mu$-almost all $\xi \in H$ we have
\begin{align*}
F(\xi,w) = F_{\mu}(\xi,w) \quad \text{for $\mathbb{P}^Q$-almost all $w \in \mathbb{W}_0(\bar{U})$.}
\end{align*}
Here $\overline{\calb(H) \otimes \calb(\mathbb{W}_0(\bar{U}))}^{\mu \otimes \mathbb{P}^Q}$ denotes the completion of $\calb(H) \otimes \calb(\mathbb{W}_0(\bar{U}))$ with respect to $\mu \otimes \mathbb{P}^Q$, and $\mathbb{P}^Q$ denotes the distribution of the $Q$-Wiener process $\bar{W}$ on $(\mathbb{W}_0(\bar{U}),\calb(\mathbb{W}_0(\bar{U})))$. Of course, $F_{\mu}$ is $\mu \otimes \mathbb{P}^Q$-almost everywhere uniquely determined.
\end{definition}

\begin{definition}[Mild solution]\label{def-mild-sol}
A martingale solution $(\bbb,X,W)$ to (\ref{SPDE}) is called a \emph{mild solution} if there exists a mapping $F \in \hat{\cale}(H)$ such that the following conditions are satisfied:
\begin{enumerate}
\item For all $\xi \in H$ and $t \in \mathbb{R}_+$ the mapping 
\begin{align*}
\mathbb{W}_0(\bar{U}) \rightarrow \mathbb{W}(H), \quad w \mapsto F(\xi,w)
\end{align*}
is $\overline{\calb_t(\mathbb{W}_0(\bar{U}))}^{\mathbb{P}^Q} / \calb_t(\mathbb{W}(H))$-measurable, where $\overline{\calb_t(\mathbb{W}_0(\bar{U}))}^{\mathbb{P}^Q}$ denotes the completion with respect to $\mathbb{P}^Q$ in $\calb(\mathbb{W}_0(\bar{U}))$.

\item We have $\bbp$-almost surely
\begin{align*}
X = F_{\mathbb{P} \circ X(0)^{-1}} (X(0),\bar{W}).
\end{align*}
\end{enumerate}
\end{definition}

\begin{remark}
Of course, every mild solution $(X,W)$ to the SPDE (\ref{SPDE}) is also a martingale solution.
\end{remark}

\begin{definition}[Existence of a mild solution]\label{def-has-mild-sol}
We say that the SPDE (\ref{SPDE}) has a \emph{mild solution} if there exists a mapping $F \in \hat{\cale}(H)$ such that:
\begin{enumerate}
\item Condition (1) from Definition \ref{def-mild-sol} is fulfilled.

\item For every standard $\mathbb{R}^{\infty}$-Wiener process $W$ on a stochastic basis $\mathbb{B}$ and any $\calf_0$-measurable random variable $\xi : \Omega \rightarrow H$ the pair $(X,W)$, where $X := F_{\bbp \circ \xi^{-1}}(\xi,\bar{W})$, is a martingale solution to (\ref{SPDE}) with $\mathbb{P}(X(0) = \xi) = 1$.
\end{enumerate}
\end{definition}

\begin{remark}
Suppose that the SPDE (\ref{SPDE}) has a mild solution. Then for every standard $\mathbb{R}^{\infty}$-Wiener process $W$ on a stochastic basis $\mathbb{B}$ and any $\calf_0$-measurable random variable $\xi : \Omega \rightarrow H$ the pair $(X,W)$, where $X := F_{\bbp \circ \xi^{-1}}(\xi,\bar{W})$, is a mild solution in the sense of Definition \ref{def-mild-sol}.
\end{remark}

\begin{definition}[Existence of a unique mild solution]\label{def-unique-sol}
We say that the SPDE (\ref{SPDE}) has a \emph{unique mild solution} if there exists a mapping $F \in \hat{\cale}(H)$ such that:
\begin{enumerate}
\item Condition (1) from Definition \ref{def-mild-sol} is fulfilled.

\item Condition (2) from Definition \ref{def-has-mild-sol} is fulfilled.

\item For any martingale solution $(X,W)$ to (\ref{SPDE}) we have $\bbp$-almost surely
\begin{align*}
X = F_{\mathbb{P} \circ X(0)^{-1}}(X(0),\bar{W}).
\end{align*}
\end{enumerate}
\end{definition}

\begin{remark}
Of course, if the SPDE (\ref{SPDE}) has a unique mild solution in the sense of Definition \ref{def-unique-sol}, then it also has a mild solution in the sense of Definition \ref{def-has-mild-sol}.
\end{remark}

\begin{remark}\label{rem-unique-mild-implies-joint}
If the SPDE (\ref{SPDE}) has a unique mild solution, then joint uniqueness in law for (\ref{SPDE}) holds, because for every martingale solution $(X,W)$ we have $\bbp$-almost surely
\begin{align*}
(X,\bar{W}) = \big( F_{\mathbb{P} \circ X(0)^{-1}}(X(0),\bar{W}), \bar{W} \big).
\end{align*}
\end{remark}

\begin{remark}\label{rem-sol-SDEs}
For $A = 0$ the SPDE (\ref{SPDE}) is rather a SDE, and in this case, we speak about \emph{weak solutions} rather than \emph{martingale solutions}, and we speak about \emph{strong solutions} rather than \emph{mild solutions}.
\end{remark}

\section{Infinite dimensional stochastic differential equations}\label{sec-SDEs}

In this section we provide the required results about infinite dimensional SDEs. Let $\calh$ be a separable Hilbert space, and consider the $\calh$-valued SDE (\ref{SDE}). Here $\bar{\alpha} : \mathbb{R}_+ \times \mathbb{W}(\calh) \rightarrow \calh$ and $\bar{\sigma} : \mathbb{R}_+ \times \mathbb{W}(H) \rightarrow L_2(U,H)$ are mappings satisfying the corresponding conditions from Assumption \ref{ass-1}. Then we are in the framework of the variational approach; see \cite[Example 4.1.3]{Prevot-Roeckner}. Furthermore, for a separable Hilbert space $H$ the Fr\'{e}chet space $\bbw(H)$ coincides with the Fr\'{e}chet space $\bbb$ defined in \cite{Rehmeier}, and their metrics induce the same topology, because for each $k \in \bbn$ and each $w \in \bbw(H)$ we have
\begin{align*}
\int_0^k \| w(t) \| dt + \sup_{t \in [0,k]} \| w(t) \| \leq (k+1) \sup_{t \in [0,k]} \| w(t) \|.
\end{align*}
Concerning the upcoming solution and uniqueness concepts we refer to Section \ref{sec-framework}, and in particular to Remark \ref{rem-sol-SDEs} regarding the terminology about weak and strong solutions.

\begin{theorem}\label{thm-1-Y}
Let $\nu$ be a probability measure on $(\calh,\calb(\calh))$. Suppose there exists a weak solution $(Y,W)$ to the SDE (\ref{SDE}) such that $\nu$ is the distribution of $X(0)$, and that joint uniqueness in law given $\nu$ holds for (\ref{SDE}). Then pathwise uniqueness given $\nu$ holds for (\ref{SDE}) as well.
\end{theorem}

\begin{proof}
This is a consequence of \cite[Thm. 3.1]{Rehmeier}, but we have to comment on a subtle detail. Note that in the latter result it is assumed that the SDE (\ref{SDE}) has a strong solution in the sense of Definition \ref{def-has-mild-sol}, whereas in the present result we merely assume that the SDE (\ref{SDE}) has a weak solution $(Y,W)$ in the sense of Definition \ref{def-mild-sol} such that $\nu$ is the distribution of $X(0)$. A careful inspection of the proof of \cite[Thm. 3.1]{Rehmeier} shows that the result also holds true in the form stated here.
\end{proof}

\begin{theorem}\cite[Thm. 3.2]{Rehmeier}\label{thm-2-Y}
For every $y \in \calh$ the following statements are equivalent:
\begin{enumerate}
\item[(i)] Uniqueness in law given $\delta_y$ holds for the SDE (\ref{SDE}).

\item[(ii)] Joint uniqueness in law given $\delta_y$ holds for the SDE (\ref{SDE}).
\end{enumerate}
\end{theorem}

We also require the following result about almost sure approximations of the It\^{o} integral. Let
\begin{align*}
b : \bbr_+ \times \bbw(H) \to L_2(U,\calh)
\end{align*}
be such that the corresponding conditions from Assumption \ref{ass-1} are fulfilled. We define
\begin{align*}
\bar{b} : \bbr_+ \times \bbw(H) \to L_2(Q^{1/2}(\bar{U}),\calh), \quad \bar{b}(t,x) := b(t,x) \circ J^{-1}.
\end{align*}
For each $j \in \bbn$ we introduce
\begin{align*}
\tilde{b}_j : \bbr_+ \times \bbw(H) \to L_2(Q^{1/2}(\bar{U}),\calh), \quad \tilde{b}_j(t,x) := \bar{b}(t,x) \bbI_{\big\{ \| \bar{b}(t,x) \|_{L_2(Q^{1/2}(\bar{U}),\calh)} \leq j \big\}},
\end{align*}
and for all $j,k \in \bbn$ we set
\begin{align*}
\tilde{b}_{j,k} : \bbr_+ \times \bbw(H) \to L(\bar{U},\calh), \quad \tilde{b}_{j,k}(t,x) := \sum_{i=1}^k \tilde{b}_{j}(t,x)(f_i) \langle f_i,\cdot \rangle_{\bar{U}},
\end{align*}
where $( f_i )_{i \in \bbn}$ is an orthonormal basis of $\bar{U}$ consisting of eigenvectors of $Q$. Furthermore, for all $j,k,\ell \in \bbn$ we define
\begin{align*}
\tilde{b}_{j,k,\ell} : \bbr_+ \times \bbw(H) \to L(\bar{U},\calh),\calh), \quad \tilde{b}_{j,k,\ell}(t,x) := \ell \int_{t-\frac{1}{\ell}}^t \tilde{b}_{j,k}(s,x)ds,
\end{align*}
and for all $j,k,\ell,m \in \bbn$ we define
\begin{align*}
\tilde{b}_{j,k,\ell,m} : \bbr_+ \times \bbw(H) \to L(\bar{U},\calh), \quad \tilde{b}_{j,k,\ell,m}(t,x) := \tilde{b}_{j,k,\ell} \bigg( \frac{[mt]}{m},x \bigg).
\end{align*}

\begin{theorem}\label{thm-integral}
Let $X$ be an adapted process with paths in $\bbw(H)$, and let $W$ be a standard $\bbr^{\infty}$-Wiener process on some stochastic basis $\bbb$ such that $\mathbb{P}$-almost surely
\begin{align*}
\int_0^t \| b(s,X) \|_{L_2(U,\calh)}^2 ds < \infty \quad \text{for all $t \in \mathbb{R}_+$.}
\end{align*}
Then for each $T \in \bbn$ there is an array of subsequences $\{ j_{T,n},k_{T,n},\ell_{T,n},m_{T,n} \}_{n \in \bbn}$ such that $I^n(X,\bar{W}) \overset{\text{a.s.}}{\to} I$ in $\bbw(\calh)$, where
\begin{align}\label{J-def}
I := \int_0^{\bullet} b(s,X) dW(s),
\end{align}
and where for all $(x,w) \in \bbw(H) \times \bbw_0(\bar{U})$ the sequence $(I^n(x,w))_{n \in \bbn}$ is defined as
\begin{equation}\label{Jn-def}
\begin{aligned}
I^n(x,w) &:= \sum_{T=1}^{\infty} \sum_{i=1}^{\infty} \tilde{b}_{j_{T,n},k_{T,n},\ell_{T,n},m_{T,n}} \bigg( \frac{i-1}{m_{T,n}}, x \bigg) \big( w^{i/m_{T,n}} - w^{(i-1)/m_{T,n}} \big)
\\ &\qquad\qquad\quad \bbI_{\{ (i-1)/m_{T,n} < T \leq i/m_{T,n} \}}, \quad n \in \bbn. 
\end{aligned}
\end{equation}
\end{theorem}

\begin{proof}
Taking into account statement (b) on page 29 in \cite{Atma-book}, the proof analogous to that of \cite[Thm. 3.1]{Chow}.
\end{proof}

\section{Proofs of the main results}\label{sec-proof}

In this section we provide the proofs of Theorems \ref{thm-1}, \ref{thm-2} and \ref{thm-3}. The general framework is that of Section~\ref{sec-framework}. In particular, we suppose that the coefficients $\alpha$ and $\sigma$ satisfy Assumption~\ref{ass-1}. In order to use the mentioned ``method of the moving frame'' from \cite{SPDE}, we require the following assumption on the semigroup $(S_t)_{t \geq 0}$.

\begin{assumption}\label{ass-2}
We suppose there exist another separable Hilbert space $\calh$, a $C_0$-group $(U_t)_{t \in \mathbb{R}}$ on $\calh$ and an isometric embedding $\ell \in L(H,\calh)$ such that the diagram
\[ \begin{CD}
\calh @>U_t>> \calh\\
@AA\ell A @VV\pi V\\
H @>S_t>> H
\end{CD} \]
commutes for every $t \in \mathbb{R}_+$, that is
\begin{align}\label{diagram-commutes}
\pi U_t \ell = S_t \quad \text{for all $t \in \mathbb{R}_+$,}
\end{align}
where $\pi := \ell^*$ is the orthogonal projection from $\calh$ into $H$.
\end{assumption}

\begin{remark}\label{rem-pseudo-contractive}
According to \cite[Prop. 8.7]{SPDE}, this assumption is satisfied if the semigroup $(S_t)_{t \geq 0}$ is \emph{pseudo-contractive}\footnote{The notion \emph{quasi-contractive} is also used in the literature.}, that is, there is a constant $\omega \in \mathbb{R}$ such that
\begin{align*}
\| S_t \| \leq e^{\omega t} \quad \text{for all $t \geq 0$.}
\end{align*}
This result relies on the Sz\H{o}kefalvi-Nagy theorem on unitary dilations (see e.g. \cite[Thm. I.8.1]{Nagy}, or \cite[Sec. 7.2]{Davies}). In the spirit of \cite{Nagy}, the group $(U_t)_{t \in \mathbb{R}}$ is called a \emph{dilation} of the semigroup $(S_t)_{t \geq 0}$.
\end{remark}

Note that $\ran(\ell)$ is a closed subspace of $\calh$, and that ${\rm ker}(\pi) = {\ran} (\ell)^{\perp}$. Therefore, we have the direct sum decomposition $\calh = \calh_1 \oplus \calh_2$, where $\calh_1 = \ran(\ell)$ and $\calh_2 = \ker(\pi)$. Furthermore, $\pi \ell$ is the identity operator and $\ell \pi$ is the orthogonal projection on $\calh_1$.  We define the continuous linear mapping $U : \bbw(\calh) \to \bbw(\calh)$ as
\begin{align*}
(Uw)(t) := U_t w(t), \quad t \in \bbr_+.
\end{align*}
Then $U$ is a linear isomorphism with inverse $U^{-1} : \bbw(\calh) \to \bbw(\calh)$ given by
\begin{align*}
(U^{-1}w)(t) := U_{-t} w(t), \quad t \in \bbr_+.
\end{align*}
Moreover, we define the continuous linear mapping
\begin{align*}
\Gamma : \mathbb{W}(\calh) \rightarrow \mathbb{W}(H), \quad \Gamma(w) := \pi U (w - \Pi_{\calh_2} w(0)),
\end{align*}
where $\Pi_{\calh_2} : \calh \to \calh_2$ denotes the orthogonal projection on $\calh_2$.

\begin{remark}\label{rem-start-in-range}
For each $w \in \mathbb{W}(\calh)$ with $w(0) \in \calh_1$ we have $\Gamma(w) = \pi U w$.
\end{remark}

\begin{lemma}\label{lemma-ran-ker-Gamma}
We have $\ran(\Gamma) = \bbw(H)$ and $\ker(\Gamma) = \bbw(\calh_2)$.
\end{lemma}

\begin{proof}
Let $v \in \bbw(H)$ be arbitrary, and set $w := U^{-1} \ell v \in \bbw(\calh)$. Then we have
\begin{align*}
\Gamma w = \pi U (w - \Pi_{\calh_2} w(0)) = \pi U (U^{-1} \ell v - \Pi_{\calh_2} \ell v(0) ) = v,
\end{align*}
showing that $\ran(\Gamma) = \bbw(H)$. Furthermore, for every $w \in \bbw(\calh)$ we have $\Gamma(w) = 0$ if and only if
\begin{align*}
U(w - \Pi_{\calh_2}w(0)) \in \bbw(\calh_2),
\end{align*}
which is equivalent to
\begin{align*}
w - \Pi_{\calh_2}w(0) \in U^{-1} \bbw(\calh_2).
\end{align*}
This condition implies that $w(0) - \Pi_{\calh_2} w(0) \in \calh_2$, which is the case if and only if $w(0) \in \calh_2$. Noting that $U^{-1} \bbw_0(\calh_2) ) = \bbw_0(\calh_2)$, we obtain
\begin{align*}
\ker(\Gamma) = \calh_2 \oplus U^{-1} \bbw_0(\calh_2) = \bbw(\calh_2),
\end{align*}
completing the proof.
\end{proof}

\begin{lemma}\label{lemma-direct-sum}
We have the direct sum decomposition $\bbw(\calh) = \bbv_1 \oplus \bbv_2$, where the subspaces $\bbv_1$ and $\bbv_2$ are given by
\begin{align*}
\bbv_1 = \bbw(\calh_1) \quad \text{and} \quad \bbv_2 = \ker(\Gamma).
\end{align*}
\end{lemma}

\begin{proof}
This follows from Lemma \ref{lemma-ran-ker-Gamma}, because $\bbw(\calh) = \bbw(\calh_1) \oplus \bbw(\calh_2)$.
\end{proof}

\begin{remark}\label{rem-Pi-1}
Consequently, the restriction $\Gamma|_{\bbv_1}$ is one-to-one with $\ran(\Gamma|_{\bbv_1}) = \bbw(H)$. Denoting by $\Delta : \bbw(H) \to \bbv_1$ its inverse, for all $w \in \bbw(\calh)$ and $v \in \bbw(H)$ we have $\Gamma(w) = v$ if and only if $\Pi_1 w = \Delta(v)$, where $\Pi_1 : \bbw(\calh) \to \bbv_1$ denotes the corresponding projection on the first coordinate. 
\end{remark}

Now, we introduce several mappings, namely
\begin{align*}
&a : \mathbb{R}_+ \times \mathbb{W}(H) \rightarrow \calh, \quad a(t,w) := U_{-t} \ell \alpha(t,w),
\\ &b : \mathbb{R}_+ \times \mathbb{W}(H) \rightarrow L_2(U,\calh), \quad b(t,w) := U_{-t} \ell \sigma(t,w),
\\ &\bar{\alpha} : \mathbb{R}_+ \times \mathbb{W}(\calh) \rightarrow \calh, \quad \bar{\alpha}(t,w) := a(t,\Gamma(w)),
\\ &\bar{\sigma} : \mathbb{R}_+ \times \mathbb{W}(\calh) \rightarrow L_2(U,\calh), \quad \bar{\sigma}(t,w) := b(t,\Gamma(w)).
\end{align*}

According to \cite[Lemmas 3.5 and 3.6]{Tappe-YW} these mappings satisfy the corresponding conditions from Assumption \ref{ass-1}. Hence, we may apply the results from Section \ref{sec-SDEs} to the SDE (\ref{SDE}). We proceed with some auxiliary results about the connections between solutions to the SPDE (\ref{SPDE}) and the SDE (\ref{SDE}).

\begin{lemma}\label{lemma-XY}
Let $\mu$ be a probability measure on $(H,\calb(H))$, let $(X,W)$ be a martingale solution to the SPDE (\ref{SPDE}) such that $\mu$ is the distribution of $X(0)$, and set 
\begin{align}\label{Y-given-by-X}
Y := \ell X(0) + \int_0^{\bullet} a(s,X) ds + \int_0^{\bullet} b(s,X) dW(s).
\end{align}
Then the following statements are true:
\begin{enumerate}
\item[(1)] The pair $(Y,W)$ is a martingale solution to the SDE (\ref{SDE}) such that $\mu \circ \ell^{-1}$ is the distribution of $Y(0)$.

\item[(2)] We have $\bbp$-almost surely $X = \Gamma(Y)$.
\end{enumerate}
\end{lemma}

\begin{proof}
This is a consequence of \cite[Cor. 3.9]{Tappe-YW}.
\end{proof}

\begin{lemma}\label{lemma-YX}
Let $\mu$ be a probability measure on $(H,\calb(H))$, let $(Y,W)$ be a martingale solution to the SDE (\ref{SDE}) such that $\mu \circ \ell^{-1}$ is the distribution of $Y(0)$, and set $X := \Gamma(Y)$. Then the following statements are true:
\begin{enumerate}
\item[(1)] The pair $(X,W)$ is a martingale solution to the SPDE (\ref{SPDE}) such that $\mu$ is the distribution of $X(0)$.

\item[(2)] We have $\bbp$-almost surely $\Pi_1 Y = \Delta(X)$ and
\begin{align*}
Y = \ell X(0) + \int_0^{\bullet} a(s,X) ds + \int_0^{\bullet} b(s,X) dW(s).
\end{align*}
\end{enumerate}
\end{lemma}

\begin{proof}
This is a consequence of \cite[Cor. 3.11]{Tappe-YW} and Remark \ref{rem-Pi-1}.
\end{proof}

Now, we prepare the required results for the proof of Theorem \ref{thm-1}.

\begin{proposition}\label{prop-1-a}
Let $(X,W)$ be a mild solution to the SPDE (\ref{SPDE}). Then the pair $(Y,W)$ with the process $Y$ given by (\ref{Y-given-by-X}) is a strong solution to the SDE (\ref{SDE}).
\end{proposition}

\begin{proof}
By Lemma \ref{lemma-XY} the pair $(Y,W)$ is a martingale solution to (\ref{SDE}). Since $(X,W)$ is a mild solution to the SPDE (\ref{SPDE}), there exists a mapping $F \in \hat{\cale}(H)$ such that the two conditions from Definition \ref{def-mild-sol} are fulfilled. Furthermore, by Theorem \ref{thm-integral} for each $T \in \bbn$ there is an array of subsequences $\{ j_{T,n},k_{T,n},\ell_{T,n},m_{T,n} \}_{n \in \bbn}$ such that $I^n(X,\bar{W}) \overset{\text{a.s.}}{\to} I$ in $\bbw(\calh)$, where $I$ is given by (\ref{J-def}), and where for all $(x,w) \in \bbw(H) \times \bbw_0(\bar{U})$ the sequence $(I^n(x,w))_{n \in \bbn}$ is given by (\ref{Jn-def}). We define a mapping 
\begin{align*}
\Phi : H \times \bbw(H) \times \bbw_0(\bar{U}) \to \bbw(\calh) 
\end{align*}
as follows. Let $\xi \in H$ and $x \in \bbw(H)$ be arbitrary. For all $w \in \bbw_0(\bar{U})$ such that $I^n(x,w)$ converges in $\bbw(\calh)$, we set
\begin{align*}
\Phi(\xi,x,w) := \ell \xi + \int_0^{\bullet} a(s,x)ds + \lim_{n \to \infty} I^n(x,w),
\end{align*}
and otherwise we set $\Phi(\xi,x,w) := 0$. Then the mapping
\begin{align}\label{Phi-meas}
\text{$\Phi$ is $\calb(H) \otimes \calb(\bbw(H)) \otimes \calb(\mathbb{W}_0(\bar{U})) / \calb(\mathbb{W}(\calh))$-measurable,}
\end{align}
and for all $\xi \in H$ and $t \in \bbr_+$ the mapping
\begin{align}\label{Phi-meas-t}
\text{$(x,w) \mapsto \Phi(\xi,x,w)$ is $\calb_t(\bbw(H)) \otimes \calb_t(\mathbb{W}_0(\bar{U})) / \calb_t(\mathbb{W}(\calh))$-measurable.}
\end{align}
Furthermore, we define the mapping
\begin{align*}
G : \calh \times \bbw_0(\bar{U}) \to \bbw(\calh), \quad G(\eta,w) := \Phi ( \pi \eta, F(\pi \eta,w), w).
\end{align*}
We claim that $G \in \hat{\cale}(\calh)$. For this purpose, for each probability measure $\nu$ on $(\calh,\calb(\calh))$ we define the mapping
\begin{align*}
G_{\nu} : \calh \times \bbw_0(\bar{U}) \to \bbw(\calh), \quad G_{\nu}(\eta,w) := \Phi ( \pi \eta, F_{\nu \circ \pi^{-1}}(\pi \eta,w), w),
\end{align*}
where the mapping $F_{\nu \circ \pi^{-1}}$ stems from Definition \ref{def-cale}. By (\ref{Phi-meas}) the mapping
\begin{align*}
\text{$G_{\nu}$ is $\overline{\calb(\calh) \otimes \calb(\mathbb{W}_0(\bar{U}))}^{\mu \otimes \mathbb{P}^Q} / \calb(\mathbb{W}(\calh))$-measurable.}
\end{align*}
Let $\nu$ be an arbitrary probability measure on $(\calh,\calb(\calh))$, and define $\mu := \nu \circ \pi^{-1}$. Since $F \in \hat{\cale}(H)$, there is a $\mu$-nullset $N \subset H$ such that for all $\xi \in N^c$ we have
\begin{align*}
F(\xi,w) = F_{\mu}(\xi,w) \quad \text{for $\mathbb{P}^Q$-almost all $w \in \mathbb{W}_0(\bar{U})$.}
\end{align*}
The set $\{ \pi \in N \} \subset \calh$ is a $\nu$-nullset, and for all $\eta \in \{ \pi \in N \}^c = \{ \pi \in N^c \}$ we have
\begin{align*}
F(\pi \eta,w) = F_{\nu \circ \pi^{-1}}(\pi \eta,w) \quad \text{for $\mathbb{P}^Q$-almost all $w \in \mathbb{W}_0(\bar{U})$.}
\end{align*}
Therefore, for $\nu$-almost all $\eta \in \calh$ we have
\begin{align*}
G(\eta,w) = G_{\nu}(\eta,w) \quad \text{for $\mathbb{P}^Q$-almost all $w \in \mathbb{W}_0(\bar{U})$,}
\end{align*}
showing that $G \in \hat{\cale}(\calh)$. Next, we show that $(Y,W)$ is a strong solution to the SDE (\ref{SDE}). By (\ref{Phi-meas-t}) for all $\eta \in \calh$ and $t \in \bbr_+$ the mapping
\begin{align*}
\text{$w \mapsto G(\eta,w)$ is $\overline{\calb_t(\mathbb{W}_0(\bar{U}))}^{\mathbb{P}^Q} / \calb_t(\mathbb{W}(\calh))$-measurable.} 
\end{align*}
Since $I^n(X,\bar{W}) \overset{\text{a.s.}}{\to} I$ in $\bbw(\calh)$, by (\ref{Y-given-by-X}) we obtain $\bbp$-almost surely
\begin{align*}
Y &= \ell X(0) + \int_0^{\bullet} a(s,X) ds + \int_0^{\bullet} b(s,X) dW(s)
\\ &= \Phi(X(0),X,\bar{W}) = \Phi \big( X(0),F_{\bbp \circ X(0)^{-1}}(X(0),\bar{W}),\bar{W} \big)
\\ &= \Phi \big( \pi Y(0), F_{\bbp \circ (\pi Y(0))^{-1}}(\pi Y(0),\bar{W}),\bar{W} \big)
\\ &= \Phi \big( \pi Y(0), F_{(\bbp \circ Y(0)^{-1}) \circ \pi^{-1}}(\pi Y(0),\bar{W}),\bar{W} \big)
\\ &= G_{\bbp \circ Y(0)^{-1}}(Y(0),\bar{W}),
\end{align*}
completing the proof.
\end{proof}

\begin{proposition}\label{prop-1-b}
Let $\mu$ be a probability measure on $(H,\calb(H))$. If joint uniqueness in law given $\mu$ holds for the SPDE (\ref{SPDE}), then joint uniqueness in law given $\mu \circ \ell^{-1}$ holds for the SDE (\ref{SDE}).
\end{proposition}

\begin{proof}
This is a consequence of Lemma \ref{lemma-YX}.
\end{proof}

\begin{proposition}\label{prop-1-c}
Let $\mu$ be a probability measure on $(H,\calb(H))$. If pathwise uniqueness given $\mu \circ \ell^{-1}$ holds for the SDE (\ref{SDE}), then pathwise uniqueness given $\mu$ holds for the SPDE (\ref{SPDE}).
\end{proposition}

\begin{proof}
This is a consequence of Lemma \ref{lemma-XY}.
\end{proof}

Now, the proof of Theorem \ref{thm-1} follows from combining Theorem \ref{thm-1-Y} and Propositions \ref{prop-1-a}--\ref{prop-1-c}. Next, we prepare the required results for the proof of Theorem \ref{thm-2}. For the following definition recall the direct sum decomposition $\bbw(\calh) = \bbv_1 \oplus \bbv_2$ from Lemma \ref{lemma-direct-sum}, and that $\Pi_1 : \bbw(\calh) \to \bbv_1$ denotes the projection on the first coordinate. 

\begin{definition}\label{def-uni-modulo}
Let $\nu$ be a probability measure on $(\calh,\calb(\calh))$. 
\begin{enumerate}
\item We say that \emph{uniqueness in law given $\nu$ modulo $\bbv_2$} holds for the SDE (\ref{SDE}) if for two weak solutions $(\bbb,Y,W)$ and $(\bbb',Y',W')$ such that
\begin{align}\label{u-in-law-mod}
\bbp \circ Y(0)^{-1} = \bbp' \circ Y'(0)^{-1} = \nu
\end{align}
as measures on $(\calh,\calb(\calh))$ we have
\begin{align*}
\bbp \circ (\Pi_1 Y)^{-1} = \bbp' \circ (\Pi_1 Y')^{-1}
\end{align*}
as measures on $(\bbv_1,\calb(\bbv_1))$.

\item We say that \emph{joint uniqueness in law given $\nu$ modulo $\bbv_2$} holds for the SDE (\ref{SDE}) if for two weak solutions $(\bbb,Y,W)$ and $(\bbb',Y',W')$ such that (\ref{u-in-law-mod}) as measures on $(\calh,\calb(\calh))$ we have
\begin{align*}
\bbp \circ (\Pi_1 Y,\bar{W})^{-1} = \bbp' \circ (\Pi_1 Y',\bar{W'})^{-1}
\end{align*}
as measures on $(\bbv_1 \times \mathbb{W}_0(\bar{U}),\calb(\bbv_1) \otimes \calb(\mathbb{W}_0(\bar{U})))$.
\end{enumerate}
\end{definition}

\begin{proposition}\label{prop-2-a}
Let $x \in H$ be arbitrary, and assume that uniqueness in law given $\delta_x$ holds for the SPDE (\ref{SPDE}). Then uniqueness in law given $\delta_{\ell(x)}$ modulo $\bbv_2$ holds for the SDE (\ref{SDE}).
\end{proposition}

\begin{proof}
This is a consequence of Lemma \ref{lemma-YX}.
\end{proof}

\begin{proposition}\label{prop-2-c}
Let $x \in H$ be arbitrary, and assume that joint uniqueness in law given $\delta_{\ell(x)}$ modulo $\bbv_2$ holds for the SDE (\ref{SDE}). Then joint uniqueness in law given $\delta_x$ holds for the SPDE (\ref{SPDE}).
\end{proposition}

\begin{proof}
This is a consequence of Lemma \ref{lemma-XY}.
\end{proof}

\begin{lemma}\label{lemma-modulo}
Let $x \in H$ be arbitrary, and assume that uniqueness in law given $\delta_x$ holds for (\ref{SPDE}). Then the following statements are equivalent:
\begin{enumerate}
\item[(i)] Joint uniqueness in law given $\delta_x$ holds for the SPDE (\ref{SPDE}).

\item[(ii)] Uniqueness in law given $\delta_{\ell(x)}$ holds for the SDE (\ref{SDE}).
\end{enumerate}
\end{lemma}

\begin{proof}
(i) $\Rightarrow$ (ii): This implication is a consequence of Lemma \ref{lemma-YX}.

\noindent (ii) $\Rightarrow$ (i): By Theorem \ref{thm-2-Y} joint uniqueness in law given $\delta_{\ell(x)}$ holds for the SDE (\ref{SDE}). Hence, this implication is a consequence of Proposition \ref{prop-2-c}.
\end{proof}

Let us sum up the previous findings (Propositions \ref{prop-2-a}, \ref{prop-2-c} and Lemma \ref{lemma-modulo}). In contrast to pathwise uniqueness and joint uniqueness in law, uniqueness in law for the SPDE (\ref{SPDE}) does not transfer to the SDE (\ref{SDE}); it only transfers modulo $\bbv_2$. Although joint uniqueness modulo $\bbv_2$ for the SDE (\ref{SDE}) implies joint uniqueness for the SPDE (\ref{SPDE}), we have seen that we necessarily need uniqueness in law for the SDE (\ref{SDE}) in order to be able to deduce joint uniqueness in law for the SPDE (\ref{SPDE}). In order to overcome these difficulties, consider the $\calh$-valued SPDE
\begin{align}\label{SPDE-Z}
dZ_t = ( \cala Z_t + \hat{\alpha}(t,Z) ) dt + \hat{\sigma}(t,Z) dW(t)
\end{align}
with coefficients
\begin{align*}
&\hat{\alpha} : \mathbb{R}_+ \times \mathbb{W}(\calh) \rightarrow \calh, \quad \hat{\alpha}(t,w) := \ell \alpha(t,\pi w),
\\ &\hat{\sigma} : \mathbb{R}_+ \times \mathbb{W}(\calh) \rightarrow L_2(U,\calh), \quad \hat{\sigma}(t,w) := \ell \sigma(t,\pi w),
\end{align*}
and where $\cala$ is the generator of the $C_0$-group $(U_t)_{t \in \bbr}$. Note that $\hat{\alpha}$ and $\hat{\sigma}$ satisfy the corresponding conditions from Assumption \ref{ass-1}.

\begin{lemma}\label{lemma-larger-space}
Let $\xi : \Omega \to H$ be a $\calf_0$-measurable random variable, and let $(Z,W)$ be a martingale solution to the SPDE (\ref{SPDE-Z}) with $Z(0) = \ell \xi$. Then $(X,W)$, where $X := \pi Z$, is a martingale solution to the SPDE (\ref{SPDE}) such that $\bbp(X(0)=\xi) = 1$.
\end{lemma}

\begin{proof}
Taking into account (\ref{diagram-commutes}), for each $t \in \bbr_+$ we have $\bbp$-almost surely
\begin{align*}
X(t) &= \pi \bigg( U_t Z(0) + \int_0^t U_{t-s} \hat{\alpha}(s,Z) ds + \int_0^t U_{t-s} \hat{\sigma}(s,Z) dW(s) \bigg)
\\ &= \pi U_t \ell \xi + \int_0^t \pi U_{t-s} \ell \alpha(s,\pi Z) ds + \int_0^t \pi U_{t-s} \ell \sigma(s,\pi Z) dW(s)
\\ &= S_t \xi + \int_0^t S_{t-s} \alpha(s,X) ds + \int_0^t S_{t-s} \sigma(s,X) dW(s),
\end{align*}
completing the proof.
\end{proof}

From now on, we assume, as in Theorem \ref{thm-2}, that the $C_0$-semigroup $(S_t)_{t \geq 0}$ can be extended to a $C_0$-group $(U_t)_{t \in \bbr}$ on $H$. More precisely, we assume there exists a $C_0$-group $(U_t)_{t \in \bbr}$ on $H$ such that $S_t = U_t$ for all $t \geq 0$. In view of Lemma \ref{lemma-larger-space}, this does not mean a severe restriction, because otherwise the SPDE can be realized on a larger state space, where this property is fulfilled. 

\begin{remark}\label{rem-v2}
In the present situation, Assumption \ref{ass-2} is satisfied with $\calh = H$ and $\ell = \pi = \Id_H$. Therefore, we have $\calh_2 = \ker(\pi) = \{ 0 \}$, and it follows that $\Gamma = U$, and in particular $\bbv_2 = \ker(\Gamma) = \{ 0 \}$.
\end{remark}

\begin{proposition}\label{prop-2-aa}
Let $x \in H$ be arbitrary, and assume that uniqueness in law given $\delta_x$ holds for the SPDE (\ref{SPDE}). Then uniqueness in law given $\delta_x$ holds for the SDE (\ref{SDE}).
\end{proposition}

\begin{proof}
This is a consequence of Proposition \ref{prop-2-a} and Remark \ref{rem-v2}.
\end{proof}

Now, the proof of Theorem \ref{thm-2} follows from combining Theorem \ref{thm-2-Y}, Proposition \ref{prop-2-aa} and Proposition \ref{prop-2-c}. Finally, let us provide the proof of Theorem \ref{thm-3}. The equivalences (i) $\Leftrightarrow$ (ii) $\Leftrightarrow$ (iii) follow from \cite[Thm. 1.1]{Tappe-YW} and its proof, the implication (i) $\Rightarrow$ (iv) follows from Remark \ref{rem-unique-mild-implies-joint}, and the implication (iv) $\Rightarrow$ (ii) is a consequence of Theorem \ref{thm-1}.

\section{Examples}\label{sec-examples}

In this section we provide some examples, illustrating our previous findings.

\begin{example}
Let $H = U = \ell^2(\bbn)$ be the Hilbert space consisting of all sequences $h = (h_k)_{k \in \bbn} \subset \bbr$ such that $\sum_{k \in \bbn} |h_k|^2 < \infty$. As in \cite[Example 2.5.4]{Pazy}, let $(S_t)_{t \geq 0}$ be the semigroup given by
\begin{align*}
S_t h := (e^{-kt} h_k)_{k \in \bbn} \quad \text{for $t \geq 0$ and $h = (h_k)_{k \in \bbn} \in H$.}
\end{align*}
Then $(S_t)_{t \geq 0}$ is a $C_0$-semigroup on $H$ with infinitesimal generator $A : \cald(A) \subset H \to H$ defined on the domain
\begin{align*}
\cald(A) = \{ (h_k)_{k \in \bbn} \in H : (k h_k)_{k \in \bbn} \in H \},
\end{align*}
and given by
\begin{align*}
Ah = (-k h_k)_{k \in \bbn} \quad \text{for $h = (h_k)_{k \in \bbn} \in \cald(A)$.}
\end{align*}
Note that $(S_t)_{t \geq 0}$ is a semigroup of contractions. Therefore, by Remark \ref{rem-pseudo-contractive} we deduce that Assumption \ref{ass-2} is fulfilled. Now, we consider the SPDE
\begin{align}\label{SPDE-Tanaka}
dX(t) = A X(t) dt + \sigma(X(t)) dW(t), \quad X(0) = 0,
\end{align}
where the volatility $\sigma : H \to L_2(H)$ is defined as
\begin{align*}
\sigma(h) := \sum_{k=1}^{\infty} \frac{\sgn(h_k)}{k} \la e_k,\cdot \ra e_k, \quad h \in H.
\end{align*}
Here $(e_k)_{k \in \bbn}$ denotes the canonical orthonormal basis of $H$, and $\sgn : \bbr \to \{ 1,-1 \}$ is the sign function defined as
\begin{align*}
\sgn(x) :=
\begin{cases}
1, & \text{if $x > 0$,}
\\ -1, & \text{if $x \leq 0$.}
\end{cases}
\end{align*}
Note that for each $h \in H$ we have indeed $\sigma(h) \in L_2(H)$, because
\begin{align*}
\sum_{k=1}^{\infty} \| \sigma(h)e_k \|^2 = \sum_{k=1}^{\infty} \frac{1}{k^2} < \infty.
\end{align*}
The SPDE (\ref{SPDE-Tanaka}) may be regarded as an infinite dimensional version of Tanaka's equation with linear drift. Using our previous findings, we obtain the following results:
\begin{enumerate}
\item Uniqueness in law holds for the SPDE (\ref{SPDE-Tanaka}).

\item Pathwise uniqueness does not hold for the SPDE (\ref{SPDE-Tanaka}).

\item The SPDE (\ref{SPDE-Tanaka}) does not have a mild solution in the sense of Definition \ref{def-mild-sol}.
\end{enumerate}
In order to prove these statements, let $(X,W)$ be a martingale solution to the SPDE (\ref{SPDE-Tanaka}). Then we have
\begin{align*}
X(t) = \int_0^t S_{t-s} \sigma(X(s)) dW(s), \quad t \in \bbr_+.
\end{align*}
We define the $H$-valued process $B$ as
\begin{align*}
B_t := \int_0^t \sigma(X(s)) dW(s), \quad t \in \bbr_+.
\end{align*}
Then $B$ is a $Q$-Wiener process with covariance operator $Q \in L_1^{++}(H)$ given by
\begin{align*}
Q = \sum_{k=1}^{\infty} \frac{1}{k^2} \la e_k,\cdot \ra e_k.
\end{align*}
Indeed, noting that $\sigma(h)$ is self-adjoint for each $h \in H$, by \cite[Thm. 2.4]{Atma-book} we have
\begin{align*}
\langle\!\langle B \rangle\!\rangle_t = Qt, \quad t \in \bbr_+.
\end{align*}
Therefore, by L\'{e}vy's theorem (see \cite[Thm. 2.6]{Atma-book}) the process $B$ is a $Q$-Wiener process. Defining the self-adjoint Hilbert Schmidt operator $J \in L_2^{++}(U)$ as 
\begin{align*}
J := \sum_{k=1}^{\infty} \frac{1}{k} \la e_k,\cdot \ra e_k,
\end{align*}
we have $Q = J^2$. Now, let
\begin{align*}
\bar{W} := \sum_{k=1}^{\infty} \beta_k J e_k = \sum_{k=1}^{\infty} \frac{1}{k} \beta_k e_k
\end{align*}
be the associated $Q$-Wiener process. Then by \cite[Prop. 2.4.5]{Liu-Roeckner} we have
\begin{align*}
X(t) &= \int_0^t S_{t-s} \sigma(X(s)) \circ J^{-1} d \bar{W}(s)
\\ &= \sum_{k=1}^{\infty} \bigg( \frac{e^{-kt}}{k} \int_0^t e^{ks} \sgn(X^k(s)) d \beta^k(s) \bigg) e_k, \quad t \in \bbr_+,
\end{align*}
and hence, for each $k \in \bbn$ we obtain
\begin{align*}
X^k(t) = \frac{e^{-kt}}{k} \int_0^t e^{ks} \sgn(X^k(s)) d \beta^k(s), \quad t \in \bbr_+.
\end{align*}
Therefore, pathwise uniqueness does not hold for the SPDE (\ref{SPDE-Tanaka}), because $(-X,W)$ is also a martingale solution. Using \cite[Prop. 2.4.5]{Liu-Roeckner} again, a similar calculation as above shows that for each $k \in \bbn$ we have
\begin{align*}
B^k(t) = \frac{1}{k} \int_0^t \sgn(X^k(s)) d \beta^k(s), \quad t \in \bbr_+.
\end{align*}
Hence, using the associativity of the It\^{o} integral and integration by parts, for each $k \in \bbn$ we obtain
\begin{align*}
X^k(t) &= e^{-kt} \int_0^t e^{ks} dB^k(s)
\\ &= e^{-kt} \bigg( e^{kt} B^k(t) - \int_0^t B^k(s) k e^{ks} ds \bigg)
\\ &= B^k(t) - k e^{-kt} \int_0^t e^{ks} B^k(s) ds, \quad t \in \bbr_+.
\end{align*}
Therefore, we have $X = \Phi(B)$ with a measurable map $\Phi : \bbw(H) \to \bbw(H)$, proving uniqueness in law. Consequently, by Theorem \ref{thm-1} the SPDE (\ref{SPDE-Tanaka}) does not have a mild solution.
\end{example}

Now, we consider SPDEs of the type
\begin{align}\label{SPDE-prog}
dX(t) = \big( A X(t) + \alpha(t,X(t)) \big) dt + \sigma(t,X(t)) dW(t)
\end{align}
with progressively measurable coefficients
\begin{align}\label{prog-coeff}
\alpha : \bbr_+ \times H \times \Omega \to H \quad \text{and} \quad \sigma : \bbr_+ \times H \times \Omega \to L_2(U,H).
\end{align}

\begin{proposition}\label{prop-delta-uniq}
Suppose that $(S_t)_{t \geq 0}$ can be extended to a unitary $C_0$-group $(U_t)_{t \in \bbr}$. Furthermore, we assume that for each $n \in \bbn$ there is a measurable function $L(n) : \bbr_+ \to \bbr_+$ with $L_T^*(n) := \sup_{t \in [0,T]} L_t(n) < \infty$ for each $T \in \bbr_+$ such that for all $(t,\omega) \in \bbr_+ \times \Omega$, all $n \in \bbn$ and all $x,y \in H$ with $\| x \|_H, \| y \|_H \leq n$ we have
\begin{equation}\label{mon}
\begin{aligned}
&2 \la x-y,\alpha(t,x,\omega)-\alpha(t,y,\omega) \ra_H
\\ &\quad + \| \sigma(t,x,\omega)-\sigma(t,y,\omega) \|_{L_2(U,H)}^2 \leq L_t(n) \| x-y \|_H^2.
\end{aligned}
\end{equation}
Then joint $\delta$-uniqueness in law holds for the SPDE (\ref{SPDE-prog}).
\end{proposition}

\begin{proof}
Since the group $(U_t)_{t \in \mathbb{R}}$ is unitary, we have $U_{-t} = U_t^*$ and $U_t$ is an isometry for every $t \in \bbr$. We consider the $H$-valued SDE
\begin{align}\label{SDE-prog}
dY(t) = \bar{\alpha}(t,Y(t))dt + \bar{\sigma}(t,X(t)) dW(t),
\end{align}
where the progressively measurable coefficients
\begin{align*}
\bar{\alpha} : \bbr_+ \times H \times \Omega \to H \quad \text{and} \quad \bar{\sigma} : \bbr_+ \times H \times \Omega \to L_2(U,H)
\end{align*}
are given by
\begin{align*}
\bar{\alpha}(t,x,\omega) := U_t^* \alpha(t,U_t x,\omega) \quad \text{and} \quad \bar{\sigma}(t,x,\omega) := U_t^* \sigma(t,U_t x,\omega).
\end{align*}
Then $\delta$-pathwise uniqueness holds for the SDE (\ref{SDE-prog}). Indeed, by (\ref{mon}) for all $(t,\omega) \in \bbr_+ \times \Omega$, all $n \in \bbn$ and all $x,y \in H$ with $\| x \|_H, \| y \|_H \leq n$ we have
\begin{equation}\label{mon-2}
\begin{aligned}
&2 \la x-y,\bar{\alpha}(t,x,\omega)-\bar{\alpha}(t,y,\omega) \ra_H + \| \bar{\sigma}(t,x,\omega)-\bar{\sigma}(t,y,\omega) \|_{L_2(U,H)}^2
\\ &= 2 \la x-y,U_t^* \alpha(t,U_t x,\omega)-U_t^* \alpha(t,U_t y,\omega) \ra_H
\\ &\quad + \| U_t^* \sigma(t,U_t x,\omega)-U_t^* \sigma(t,U_t y,\omega) \|_{L_2(U,H)}^2
\\ &= 2 \la U_t x - U_t y, \alpha(t,U_t x,\omega) - \alpha(t,U_t y,\omega) \ra_H
\\ &\quad + \| \sigma(t,U_t x,\omega) - \sigma(t,U_t y,\omega) \|_{L_2(U,H)}^2
\\ &\leq L_t(n) \| U_t x - U_t y \|_H^2 = L_t(n) \| x-y \|_H^2.
\end{aligned}
\end{equation}
Now, let $(Y,W)$ and $(Y',W)$ be two weak solutions to the SDE (\ref{SDE-prog}) on the same stochastic basis $\bbb$ such that $\bbp(Y(0) = Y'(0)) = 1$ and $\delta_y$ is the distribution of $Y(0)$ for some $y \in H$. A standard procedure using It\^{o}'s formula (see, for example \cite[Thm. 4.2.5]{Liu-Roeckner}) and inequality (\ref{mon-2}) shows for any fixed $T \in \bbr_+$ and each $n \in \bbn$ with $n \geq \| y \|_H$ the inequality
\begin{align*}
&\bbe \big[ \| Y(t \wedge T_n) - Y'(t \wedge T_n) \|_H^2 \big]
\\ &\leq L_T^*(n) \int_0^t \bbe \big[ \| Y(s \wedge T_n) - Y'(s \wedge T_n) \|_H^2 \big] ds, \quad t \in [0,T],
\end{align*}
where $(T_n)_{n \in \bbn}$ denotes the localizing sequence of stopping times
\begin{align*}
T_n := \inf \{ t \in \bbr_+ : \| Y(t) \|_H \geq n \} \wedge \inf \{ t \in \bbr_+ : \| Y'(t) \|_H \geq n \}.
\end{align*}
Therefore, by Gronwall's lemma we obtain $\bbp$-almost surely $Y = Y'$. Consequently, taking into account Remark \ref{rem-start-in-range}, by Proposition \ref{prop-1-c} we deduce that that $\delta$-pathwise uniqueness holds for the SPDE (\ref{SPDE-prog}). Therefore, by Theorem \ref{thm-2} joint $\delta$-uniqueness in law holds for the SPDE (\ref{SPDE-prog}).
\end{proof}

\begin{example}
We consider a SPDE of the form
\begin{align}\label{SPDE-diff}
dX(t) = \bigg( \frac{d}{d x} X(t) + \alpha(t,X(t)) \bigg) dt + \sigma(t,X(t)) dW(t)
\end{align}
on the state space $H = L^2(\bbr)$ with progressively measurable coefficients (\ref{prog-coeff}). The differential operator $\frac{d}{dx}$ is generated by the translation semigroup $(S_t)_{t \geq 0}$ given by
\begin{align*}
S_t h := h(t + \bullet) \quad \text{for $t \geq 0$ and $h \in H$.}
\end{align*}
The semigroup $(S_t)_{t \geq 0}$ extends to a unitary $C_0$-group $(U_t)_{t \in \bbr}$ by setting
\begin{align*}
U_t h := h(t + \bullet) \quad \text{for $t \in \bbr$ and $h \in H$.}
\end{align*}
Note that for all $t \in \bbr_+$ and all $h,g \in H$ we have
\begin{align*}
\la U_t h,g \ra = \int_{\bbr} h(x+t)g(x) dx = \int_{\bbr} h(x)g(x-t) dx = \la h, U_{-t} g \ra,
\end{align*}
showing that the group $(U_t)_{t \in \bbr}$ is indeed unitary. Suppose that the coefficients satisfy condition (\ref{mon}). Then, according to Proposition \ref{prop-delta-uniq}, joint $\delta$-uniqueness in law holds for the SPDE (\ref{SPDE-diff}).
\end{example}

Before we proceed, let us emphasize that all the results of this paper also hold true for SPDEs of the type
\begin{align}\label{SPDE-trace-class}
dX(t) = (A X(t) + \alpha(t,X)) dt + \bar{\sigma}(t,X)d\bar{W}(t).
\end{align}
driven by a trace class Wiener process $\bar{W}$ with some covariance operator $Q \in L_1^{++}(U)$ and path-dependent coefficients $\alpha$ and $\bar{\sigma}$ such that $\bar{\sigma}$ is $L_2(Q^{1/2}(U),H)$-valued, and $\alpha$ and $\bar{\sigma}$ satisfy the corresponding conditions from Assumption \ref{ass-1}. Indeed, setting $J := Q^{1/2}$ we have $J \in L_2^{++}(U)$. There are an orthonormal basis $(e_k)_{k \in \bbn}$ of $U$ and a sequence $(\lambda_k)_{k \in \bbn} \subset (0,\infty)$ such that $\lambda_k \downarrow 0$ and
\begin{align*}
Q e_k = \lambda_k e_k, \quad k \in \bbn.
\end{align*}
The $Q$-Wiener process $\bar{W}$ admits the series representation
\begin{align*}
\bar{W} = \sum_{k=1}^{\infty} \sqrt{\lambda_k} \beta_k e_k = \sum_{k=1}^{\infty} \beta_k Je_k,
\end{align*}
where the sequence $(\beta_k)_{k \in \bbn}$ given by
\begin{align*}
\beta_k = \frac{1}{\sqrt{\lambda_k}} \la \bar{W},e_k \ra, \quad k \in \bbn
\end{align*}
consists of independent Wiener processes; see, for example \cite[Prop. 4.3]{Da_Prato}. By Remark \ref{rem-integral-cylindrical} the stochastic integral is given by
\begin{align*}
\int_0^t S_{t-s} \sigma(s,X) d \bar{W}(s) = \int_0^t S_{t-s} \sigma(s,X) \circ J \, dW(s), \quad t \in \bbr_+,
\end{align*}
where $W$ denotes the $\bbr^{\infty}$-standard Wiener process $W = (\beta_k)_{k \in \bbn}$.
Therefore, we can rewrite the SPDE (\ref{SPDE-trace-class}) equivalently as
\begin{align*}
dX(t) = (A X(t) + \alpha(t,X)) dt + \bar{\sigma}(t,X) \circ Q^{1/2} dW(t),
\end{align*}
which is of the type (\ref{SPDE}) with the $L_2(U,H)$-valued mapping $\sigma$ given by
\begin{align*}
\sigma(t,w) = \bar{\sigma}(t,w) \circ Q^{1/2}, \quad (t,w) \in \bbr_+ \times \bbw(H).
\end{align*}
There are several well-known situations, where Theorem \ref{thm-3} applies, and, as a consequence, there exists a unique mild solution and joint uniqueness in law holds. Let us outline some of these situations:
\begin{enumerate}
\item Let $W$ be a standard $\bbr^{\infty}$-Wiener process. Assume that Hypotheses 7.2 and condition (7.31) from \cite{Da_Prato} are satisfied. Then Theorem \ref{thm-3} applies to the SPDE (\ref{SPDE}).

\item Let $W$ be a standard $\bbr^{\infty}$-Wiener process. Assume that $(S_t)_{t \geq 0}$ is a semigroup of contractions, and that the coefficients $\alpha$ and $\sigma$ satisfy local monotonicity and coercivity conditions; see \cite[Thm. 2.6]{Tappe-mon} for details. Then Theorem \ref{thm-3} applies to the SPDE (\ref{SPDE}).

\item Let $\bar{W}$ be a trace class Wiener process. If Lipschitz and linear growth conditions are fulfilled (see, for example Hypothesis 7.1 in \cite{Da_Prato}), then Theorem \ref{thm-3} applies to the SPDE (\ref{SPDE-trace-class}). Slightly more general, we can also impose locally Lipschitz and linear growth conditions; see, for example \cite{Tappe-refine}.

\item Let $\bar{W}$ be a trace class Wiener process. If the semigroup $(S_t)_{t \geq 0}$ is compact and the coefficients $\alpha$ and $\sigma$ are continuous and satisfy the linear growth condition, then the existence of martingale solutions to (\ref{SPDE-trace-class}) holds true; see, for example \cite[Thm. 3.14]{Atma-book}. If, moreover, uniqueness in law holds true, then Theorem \ref{thm-3} applies to the SPDE (\ref{SPDE-trace-class}); cf. also \cite[Sec. 4]{Tappe-YW}.
\end{enumerate}

\end{document}